\documentclass[fleqn,final]{zzz}
\usepackage[pdftex]{hyperref}
\usepackage{enumerate}

\hypersetup{
   colorlinks = true,
   urlcolor = blue,
   linkcolor = red,
   citecolor = green
}
\usepackage{mathrsfs}
\usepackage{color}
\usepackage{tikz}
\usepackage{float}
\usepackage{graphicx}
\usepackage{rotating}
\newcommand{\hyp}[5]{\,\mbox{}_{#1}F_{#2}\!\left(
  \genfrac{}{}{0pt}{}{#3}{#4};#5\right)}

\newcommand\TTTT{\rule{0pt}\c@equation\m@ne$$\global\@ignoretrue}

\def\@yeqncr{\@ifnextchar [{\@xeqncr}{\@xeqncr[5pt]}}
\makeatother

\newcommand{\LW}{{V}}
\newcommand{\expe}{{\mathrm e}}

\newcommand{\R}{\mathbb{R}}
\newcommand{\C}{\mathbb{C}}
\newcommand{\Z}{\mathbb{Z}}
\newcommand{\N}{\mathbb{N}}

\DeclareMathOperator{\F}{{\rm F}}

\renewcommand{\r}{\mathbf{r}}

\newcommand{\0}{\mathbf{0}}
\DeclareMathOperator{\sn}{sn}
\DeclareMathOperator{\cn}{cn}
\DeclareMathOperator{\dn}{dn}
\DeclareMathOperator{\ssc}{sc}
\DeclareMathOperator{\nd}{nd}
\DeclareMathOperator{\cs}{cs}

\DeclareMathOperator{\ns}{ns}
\DeclareMathOperator{\nc}{nc}
\DeclareMathOperator{\dc}{dc}

\DeclareMathOperator{\arcsc}{arcsc}
\DeclareMathOperator{\sech}{sech}
\DeclareMathOperator{\PP}{{\sf P}}

\numberwithin{equation}{section}
\numberwithin{equation}{section}
\numberwithin{corollary}{section}
\numberwithin{remark}{section}
\numberwithin{theorem}{section}
\numberwithin{lemma}{section}

\begin{document}

\renewcommand{\PaperNumber}{***}

\FirstPageHeading

\ArticleName{Peanut harmonic expansion for a fundamental solution of Laplace's equation in flat-ring coordinates}

\ShortArticleName{Peanut harmonic expansion in flat-ring coordinates}

\Author{Lijuan Bi\,$^\ast$, Howard S. Cohl\,$^\dag\!\!\ $
and Hans Volkmer\,$^\S\!\!\ $}

\AuthorNameForHeading{L. Bi, H.~S.~Cohl, H.~Volkmer}

\Address{$^\dag$ Department of Mathematics,
The Ohio State University at Newark,
Newark, OH 43055, USA
\URLaddressD{
\href{https://newark.osu.edu/directory/bi-lijuan.html}
{https://newark.osu.edu/directory/bi-lijuan.html}
}
} 
\EmailD{bi.146@osu.edu} 

\Address{$^\dag$ Applied and Computational
Mathematics Division, National Institute of Standards
and Technology, Mission Viejo, CA 92694, USA
\URLaddressD{
\href{http://www.nist.gov/itl/math/msg/howard-s-cohl.cfm}
{http://www.nist.gov/itl/math/msg/howard-s-cohl.cfm}
}
} 
\EmailD{howard.cohl@nist.gov} 

\Address{$^\S$ Department of Mathematical Sciences,
University of Wisconsin-Milwaukee,
Milwaukee, WI 53201-0413, USA
} 
\EmailD{volkmer@uwm.edu} 

\ArticleDates{Received ?? 2021 in final form ????; Published online ????}

\Abstract{
We derive an expansion for the fundamental solution of Laplace's equation in flat-ring cyclide coordinates in three-dimensional Euclidean space. This expansion is a double series of products of functions that are harmonic in the interior and exterior of coordinate surfaces which are peanut shaped and orthogonal to surfaces which are flat-rings. These internal and external peanut harmonic functions are expressed in terms of Lam\'e-Wangerin functions. Using the expansion for the fundamental solution, we derive an addition theorem for the azimuthal Fourier component in terms of the odd-half-integer degree Legendre function of the second kind as an infinite series in Lam\'e-Wangerin functions.  We also derive integral identities over the Legendre function of the second kind 
for a product of three Lam\'e-Wangerin functions. In a limiting case we obtain the expansion of the fundamental solution in spherical coordinates.}


\Classification{35A08; 35J05; 33C05; 
33C45; 33C47; 33C55; 33C75}

\newcommand\TTT{\rule{0pt}{4.0ex}}
\newcommand\TT{\rule{0pt}{3.0ex}}
\newcommand\TB{\rule[-2.7ex]{0pt}{0pt}}
\newcommand\TTB{\rule[-3.7ex]{0pt}{0pt}}

\newcommand{\myref}[1]{(\ref{#1})}
\def\cprime{$'$}

\newtheorem{thm}[lemma]{Theorem}
\newtheorem{cor}[lemma]{Corollary}
\newtheorem{rem}[lemma]{Remark}
\newtheorem{lem}[lemma]{Lemma}
\newtheorem{conj}[lemma]{Conjecture}
\newtheorem{prop}[lemma]{Proposition}
\newtheorem{defn}[lemma]{Definition}

\makeatletter
\def\eqnarray{\stepcounter{equation}\let\@currentlabel=\theequation
\global\@eqnswtrue
\tabskip\@centering\let\\=\@eqncr
$$\halign to \displaywidth\bgroup\hfil\global\@eqcnt\z@
  $\displaystyle\tabskip\z@{##}$&\global\@eqcnt\@ne
  \hfil$\displaystyle{{}##{}}$\hfil
  &\global\@eqcnt\tw@ $\displaystyle{##}$\hfil
  \tabskip\@centering&\llap{##}\tabskip\z@\cr}

\section{Introduction}

There are 17 conformally inequivalent
triply-orthogonal curvilinear
coordinate systems $(\xi_1,\xi_2,\xi_3)$ which parametrize
points $\r:=(x,y,z)\in\R^3$ such that
\[
\r=(x(\xi_1,\xi_2,\xi_3),
y(\xi_1,\xi_2,\xi_3),
z(\xi_1,\xi_2,\xi_3)),
\]
which yield solution
by separation of variables for the
three variable Laplace equation \cite[Section 3.6]{Miller}
(see also \cite{Bocher}). These 17 coordinate systems can be divided into
several groups depending on the properties
of the two-dimensional surfaces which are obtained by setting one of the coordinates to a constant in its range.

Nine of the 17 coordinate systems are rotationally-invariant, that is they can
be written as a coordinate transformation
to Cartesian coordinates of a form
\[
\r=(R(\xi_1,\xi_2)\cos\phi,R(\xi_1,\xi_2)\sin\phi,z(\xi_1,\xi_2)),
\]
where $\phi\in\R$ or specifically $\phi\in[-\pi,\pi)$ to cover
all of $\R^3$. Of the nine rotationally-invariant coordinate systems, five of them are represented by coordinate surfaces which are quadric (cylindrical, spherical, parabolic, oblate spheroidal and prolate spheroidal) and the other four are represented by coordinate
surfaces which are cyclidic (toroidal, flat-ring, flat-disk and bi-cyclide). The study of the harmonics for the quadric coordinate systems is classical (yet not fully explored), however for the cyclidic coordinate systems much remains to be learned. In toroidal coordinates the separated solutions are given in terms of associated Legendre functions. However, in the three remaining rotationally-invariant cyclidic coordinate systems (flat-ring, flat-disk, bi-cyclide) the harmonic solutions are given in terms
of second-order ordinary differential equations with four regular singularities, namely those of the Heun-type \cite[Chapter 31]{NIST:DLMF} which specialize in the cyclidic case to ordinary differential equations
of Lam\'e-type \cite[Chapter 29]{NIST:DLMF} (Lam\'{e} functions, modified Lam\'{e} functions and Lam\'{e}-Wangerin functions).
This paper is the second in a series
of papers which will focus on the expansion
of the $1/r$ potential in the rotationally
invariant cyclide coordinate systems:~flat-ring cyclide, flat-disk cyclide and bi-cyclide coordinates.

In a previous paper \cite{BiCohlVolkmerA} we studied flat-ring coordinates (see \S3.1 for their definition)
originally introduced by Wangerin \cite{Wangerin1875}.
We introduced internal and external flat-ring harmonics. These are harmonic functions which are harmonic inside and outside of
coordinate surfaces which are flat-ring cyclides.
We found the expansion of the $1/r$ potential in terms of products
of internal and external flat-ring harmonics. 
We also showed that flat-ring coordinates become toroidal coordinates in the limit $k\to0$ and the
expansion of $1/r$ approaches its known expansion in products of internal and external toroidal harmonics as $k\to 0$.

In this paper we continue our work in flat-ring coordinates, however, we now consider a second
family of compact coordinate surfaces which are given by
rotationally-invariant ``peanut'' shaped cyclides (or simply peanut cyclides).
In Section 3 of this paper we introduce corresponding internal and external peanut harmonics, and 
find the expansion of $1/r$ in a series of products of internal and external peanut harmonics
(Theorem \ref{3:expansion}).
The major difference between the expansion over flat-ring surfaces as opposed to peanut surfaces is that in the peanut case we require Lam\'e-Wangerin functions in place of periodic Lam\'e
functions. Lam\'e-Wangerin functions are not as well-known as periodic Lam\'e functions.
Therefore, in Section 2 we start with collecting the properties of Lam\'e-Wangerin functions that we will require in our analysis.
In section 4 we show that flat-ring coordinates become spherical coordinates in the limit $k\to1$ and the 
expansion of $1/r$ approaches its well-known expansion in products of internal and external spherical harmonics  as $k\to1$.

We believe that our results from Sections 3 and 4 (and partially also
from Section 2) are new.

\section{Lam\'e-Wangerin functions}\label{sec2}

Let $K=K(k)$ and $K'=K'(k)=K(k')$, $k'=\sqrt{1-k^2}$, denote the complete elliptic integral of the first kind and its
corresponding complementary elliptic integral, respectively
\cite[(19.2.8-9)]{NIST:DLMF}.
The Lam\'e differential equation \cite[(29.2.1)]{NIST:DLMF} is
\begin{equation}\label{2:lame}
 \frac{{\mathrm d}^2w}{{\mathrm d}s^2}+(h-\nu(\nu+1)k^2\sn^2(s,k))w=0 ,
\end{equation}
where $0<k<1$, $\nu\ge \frac12$, and $h$ is the eigenvalue parameter.
This equation has regular singular points at $s=\pm iK'$ with exponents $\{-\nu,\nu+1\}$ at both points.
In the application to flat-ring coordinates we require solutions of \eqref{2:lame}
that are Fuchs-Frobenius solutions \cite[Chapter XVI]{Ince} at $s=\pm iK'$ belonging to the exponent $\nu+1$
at both points simultaneously. This leads to an eigenvalue problem for the Lam\'e equation.

To simplify notation we set $s=it$
and using Jacobi's imaginary transformation \cite[\S 22.6(iv)]{NIST:DLMF} we obtain the {\it modified Lam\'e equation}
\begin{equation}\label{2:modlame0}
\frac{{\mathrm d}^2w}{{\mathrm d}t^2}+\left(\lambda-\nu(\nu+1)\dc^2(t,k')\right)w=0,\quad \lambda=\nu(\nu+1)-h,
\end{equation}
where we used Glaisher's notation \cite[\S 22.2]{NIST:DLMF} for Jacobian elliptic functions.
Again to simplify notation, we change $k'$ back to $k$ and consider the equation
\begin{equation}\label{2:modlame}
\frac{{\mathrm d}^2w}{{\mathrm d}t^2}+\left(\lambda-\nu(\nu+1)\dc^2(t,k)\right)w=0.
\end{equation}
Equation \eqref{2:modlame} has regular singularities at the points $t=\pm K$ with exponents $\{-\nu,\mu+1\}$.
We impose
the boundary conditions that the solution $w(t)$, $-K<t<K$, belongs to the exponent $\nu+1$ at both singular end points, that is, $w(t)$ can be written in the form
\begin{equation}\label{2:bc1}
 w(t)=\sum_{j=0}^\infty a_j (K-t)^{j+\nu+1}\quad\text{for $t$ close to $K$},
\end{equation}
and in the form
\begin{equation}\label{2:bc2}
  w(t)=\sum_{j=0}^\infty b_j(t+K)^{j+\nu+1}\quad\text{for $t$ close to $-K$.}
\end{equation}
The conditions \eqref{2:bc1}, \eqref{2:bc2} are equivalent to the conditions that $|t\mp K|^{-1/2}w(t)$ stays bounded as $t\to \pm K$, respectively.
We call $\lambda$ an eigenvalue if there exists a nontrivial solution $w(t)$, $-K<t<K$, of \eqref{2:modlame}
which satisfies both boundary conditions \eqref{2:bc1}, \eqref{2:bc2}. An eigenfunction $w(t)$ is called a Lam\'e-Wangerin function \cite[(15.6)]{ErdelyiHTFIII} .

In \cite[\S 3]{Volkmer2018} it is shown that
the eigenvalues $\lambda$ are real and they form an increasing sequence indexed by $n\in\N_0$
with
\[
\lambda:=\Lambda_\nu^n(k),
\]
such that $\Lambda_\nu^n(k)\to\infty$
as $n\to\infty$.
Actually, in \cite{Volkmer2018} the differential equation
\begin{equation}\label{2:modlame2}
 \frac{{\mathrm d}^2w}{{\mathrm d}r^2}+(\lambda-\nu(\nu+1)\ns^2(r,k))w=0
 \end{equation}
is treated which agrees with \eqref{2:modlame} substituting $r=t+K$.
The (real-valued) eigenfunction corresponding to $\Lambda_\nu^n(k)$
is denoted by $w(t)=W_\nu^n(t,k)$.
The eigenfunctions are normalized according to
\begin{equation}
\int_{-K}^K \{W_\nu^n(t,k)\}^2\,{\mathrm d}t =1 .
\label{LWnorm}
\end{equation}
The function $W_\nu^n(t,k)$ has exactly $n$ zeros in $(-K,K)$ \cite[\S 8]{Volkmer2018},
and it is an even function for even $n$ and an odd function for odd $n$.
Clearly, $W_\nu^n(t,k)$ converges to $0$ as $t\to\pm K$.
We now state the completeness of the eigenfunctions \cite[\S 3]{Volkmer2018}.

\begin{thm}\label{2:t2}
The system $\{W_\nu^n(t,k)\}_{n=0}^\infty$ forms an orthonormal basis in $L^2(-K,K)$.
\end{thm}

It is useful to consider equation \eqref{2:modlame2} in the limit $k\to 0$:
\begin{equation}\label{2:lame0}
 \frac{{\mathrm d}^2w}{{\mathrm d}r^2}+(\lambda-\nu(\nu+1)\csc^2 r)w =0\quad\text{for $0<r<\pi$.}
\end{equation}
In this case the eigenvalues
are $\lambda=(n+\nu+1)^2$, $n\in\N_0$, with corresponding (not normalized) eigenfunctions
\[ w(r)=W_\nu^n(r)=(\sin r)^{\nu+1} P_n^{(\nu+\frac12,\nu+\frac12)}(\cos r)
\simeq
(\sin r)^{\nu+1}C_n^{\nu+1}(\cos r),
\]
where $0<r<\pi$, employing Jacobi polynomials ${P}_n^{(\alpha,\beta)}$
\cite[Table 18.3.1]{NIST:DLMF}
in the ultraspherical case $\alpha=\beta$
\cite[(4.24.2)]{Szego}
and Gegenbauer (ultraspherical)
polynomials $C_n^\mu$ \cite[(18.7.1)]{NIST:DLMF}.

We can compare \eqref{2:modlame2} with \eqref{2:lame0} using the following lemma.

\begin{lemma}\label{2:l1}
Let $k\in(0,1)$, $\omega:=\frac{\pi}{2K}$.
Then, for all $r\in(0,K]$, we have
\begin{eqnarray}\label{2:est1}
&&\hspace{-10cm}  \cs(r,k)\le  \omega\cot (\omega r),\label{jacobi1}\\
&&\hspace{-10cm}   \sn(r,k)\le  \omega^{-1}\sin (\omega r).\label{jacobi2}
\end{eqnarray}
\end{lemma}
\begin{proof}
(a)
The inverse function of $\ssc(r,k)$, $0\le r< K$,  is \cite[(22.15.20)]{NIST:DLMF}
\[ \arcsc(x,k):=\int_0^x \frac{{\mathrm d}t}{\sqrt{1+t^2}\sqrt{1+k'^2t^2}}\quad\text{for $x\ge 0$}.\]
Now \eqref{jacobi1} is equivalent to
\begin{equation}\label{jacobi3}
 g(x):=\omega^{-1}\arctan (\omega x)-\arcsc(x,k)\ge 0\quad\text{for $x\ge 0$.}
\end{equation}
We have $g(0)=\lim_{x\to\infty}g(x)=0$. A calculation shows that $g'(x)>0$ is equivalent
to $1+k'^2-2\omega^2+x^2(k'^2-\omega^4)>0$.
Now $k'^2-\omega^4<0$ so $g$ first increases and then decreases, establishing \eqref{jacobi3}.
Note that the inequality $\omega>\sqrt{k'}$ follows from
the fact that $\omega$ equals the arithmetic-geometric mean of $1$ and $k'$,
$M(1,k')=\frac{\pi}{2K}=\omega$ \cite[(22.20.6)]{NIST:DLMF}.\\
(b)
By squaring both sides of inequality \eqref{jacobi1} we find
\[ \ns^2(r,k)-1\le \omega^2(\csc^2 (\omega r)-1)\le \csc^2 (\omega r)-1,\]
so $\sin \omega  r\le  \sn(r,k)$ for $0\le r\le K$.
Since $\sn^2(x,k)+\cn^2(x,k)=1$, this implies
\[ \cn(r,k)\le \cos (\omega r)\quad \text{for $0\le r\le K$},\]
so
\[ \dn(r,k)\cn(r,k)\le \cos(\omega r)\quad \text{for $0\le r\le K$}.\]
Integrating from $0$ to $r$ gives  \eqref{jacobi2}.\\
\end{proof}

\begin{lemma}\label{2:l2}
Let $\omega:=\frac\pi{2K}$. If $\nu\ge 0$ then
\[ \omega^2(n+\nu+1)^2\le \Lambda_\nu^n(k)\le \nu(\nu+1)(1-\omega^2)+\omega^2(n+\nu+1)^2.\]
If $-\frac12\le \nu<0$ then
\[ \nu(\nu+1)(1-\omega^2)+\omega^2(n+\nu+1)^2\le \Lambda_\nu^n(k)\le \omega^2 (n+\nu+1)^2.\]
\end{lemma}
\begin{proof}
By \eqref{jacobi1} and \eqref{jacobi2},
\[ \omega^2\csc^2(\omega r)\le \ns^2(r,k)\le \omega^2\csc^2(\omega r)+1-\omega^2\quad\text{for $0<r<2K$}.\]
The bounds for $\Lambda_\nu^n(k)$ follow from the Sturm comparison theorem
\cite[\S 10.4]{Ince} comparing \eqref{2:modlame2} with \eqref{2:lame0}.
\end{proof}

If $\nu\ge 0$ we can also estimate $\nu(\nu+1)\ns^2(r,k)\ge \nu(\nu+1)$ which leads to
\begin{equation}\label{2:lowerbound}
\nu(\nu+1)+\omega^2(n+1)^2\le\Lambda_\nu^n(k) ,
\end{equation}

\begin{lemma}\label{2:l3}
If $\nu\ge 0$, $n\in\N_0$, $-K<t<K$ then
\begin{equation}\label{est5}
 \left\{W_\nu^n(t,k)\right\}^2\le \frac{\pi}{2K}(n+\nu+1)
\end{equation}
and
\begin{equation}\label{est6}
 \dc(t,k) \left\{W_\nu^n(t,k)\right\}^2\le \frac{\pi^2}{4K}(n+\nu+1)^2 .
\end{equation}
\end{lemma}
\begin{proof}
Note that the function $w(t):=W_\nu^n(t,k)$ is even or odd so it is enough to consider $t\in[0,K)$.
Set $q(t)=\nu(\nu+1)k'^2\ssc^2(t,k)=\nu(\nu+1)(\dc^2(t,k)-1)$ and $h=\nu(\nu+1)-\Lambda_\nu^n(k)$. By multiplying \eqref{2:modlame} by $w$ and integrating from $0$ to $t$ with $0\le t<K$, we get
\[
 w(t)w'(t)-\int_0^t w'(\tau)^2\, {\mathrm d}\tau -h\int_0^t w(\tau)^2\, {\mathrm d}\tau-\int_0^t q(\tau)w(\tau)^2\, {\mathrm d}\tau =0.\]
Since $q\ge 0$ and $h\le 0$ by \eqref{2:lowerbound}, this gives
\begin{equation}\label{est7}
 \int_0^t w'(\tau)^2\, {\mathrm d}\tau\le \tfrac12 (-h)+w(t)w'(t) .
\end{equation}
Since $w(t)(K-t)^{-\nu-1}$ is analytic at $t=K$, $w(t)w'(t)\to 0$ as $t\to K$.
Therefore, by Lemma \ref{2:l2},
\[ \int_0^K w'(\tau)^2\, {\mathrm d}\tau \le \tfrac12 (-h)\le \frac{\pi^2}{8K^2}(n+\nu+1)^2 .\]
Now \eqref{est5} follows from
\[ w(t)^2=2\int_K^t w(\tau)w'(\tau)\, {\mathrm d}\tau\le 2 \left(\int_0^K w(\tau)^2\, {\mathrm d}\tau\right)^{1/2} \left(\int_0^K w'(\tau)^2\, {\mathrm d}\tau\right)^{1/2}. \]
To get \eqref{est6} we infer from \eqref{est7} that
\[ -2w(t)w'(t)\le -h. \]
Integrating from $t$ to $K$ gives
\begin{equation}\label{est8}
 w(t)^2\le (-h)(K-t)\quad\text{for $0\le t<K$}.
\end{equation}
Since $\sn(x,k)$ is a concave function of $x\in[0,K]$,
\begin{equation}\label{ineqsn}
\sn(x,k)\ge \frac{x}{K}\quad\text{for $x\in[0,K]$,}
\end{equation}
so
\[ \dc(t,k)w(t)^2=\ns(K-t,k)w(t)^2\le K(K-t)^{-1}w(t)^2 \]
which together with \eqref{est8} yields \eqref{est6}.
\end{proof}

The case $\nu=-\frac12$ (which is not covered by the preceding lemma) is treated in the following lemma.

\begin{lemma}\label{2:l4}
There is a positive constant $C$ independent of $n$ and $t$ such that, for all $n\in\N_0$ and $-K<t<K$,
\begin{equation}\label{2:est9}
 \left\{W_{-\frac12}^n(t,k)\right\}^2\le C
 \end{equation}
and
\begin{equation}\label{2:est10}
 \dc(t,k)\{W_{-\frac12}^n(t,k)\}^2\le C(1+n)^2 .
\end{equation}
\end{lemma}
\begin{proof}
The function $w(t):=W_{-\frac12}^n(t,k)$, $t\in(-K,K)$, satisfies the differential equation
\begin{equation}\label{2:specialmodlame}
\frac{{\mathrm d}^2w}{{\mathrm d}t^2} +p(t) w =0,
\end{equation}
where
\[ p(t)=\Lambda_{-\frac12}^n(k)+\tfrac14+\tfrac14k'^2 \ssc^2(t,k) .\]
By Lemma \ref{2:l2}, $\Lambda_{-\frac12}^n(k)>-\frac14$ so $p(t)>0$ and $p'(t)>0$ for all $t\in[0,K)$.
It follows from \eqref{2:specialmodlame} that
\[ \frac{{\mathrm d}}{{\mathrm d}t}\left[w(t)^2+\frac{1}{p(t)} w'(t)^2\right]=-\frac{p'(t)}{p(t)^2} w'(t)^2\le 0\quad \text{for $t\in[0,K)$}.\]
This implies that the amplitude values of $w(t)^2$ (when $w'(t)=0$) decrease on $[0,K)$.
Since $w(t)\to 0$ as $t\to \pm K$ the maximum value of $w(t)^2$ is one of the amplitude values.
If $n$ is even the maximum value of $w(t)^2$ is $w(0)^2$. If $n$ is sufficiently large,
there is $t_0\in[\frac13K,\frac23K]$ such that $w(t_0)=0$.
Then \eqref{2:est9} follows from
\cite[Lemma 4.3]{BiCohlVolkmerA} applied to the interval $[0,t_0]$.
If $n$ is odd the proof is similar.

To prove \eqref{2:est10}, we set
\[ u(r):=\sn^{-1/2}(r,k)\, w(r-K)\quad\text{for $0<r<2K$}.\]
Then $u$ satisfies the differential equation
\begin{equation}\label{speciallame}
\frac{{\mathrm d}}{{\mathrm d}r} \left(\sn(r,k) \frac{\mathrm{d}u}{{\mathrm d}r}\right)+q(r)\sn(r,k) u=0 ,
\end{equation}
where
\[ q(r)=-h+\tfrac14k'^2-\tfrac34\dn^2(r,k)\le -h,\quad -h=\tfrac14+\Lambda_{-\frac12}^n(k)>0. \]
This differential equation has regular singular points at $r=0, 2K$ with exponents $0,0$.
This shows that $u(r)$ is analytic at the points $r=0, 2K$.
We multiply \eqref{speciallame} by $u(r)$ and integrate from $0$ to $r$ to find
\[ -\sn(r,k)u(r)u'(r)+\int_0^r \sn(\sigma,k)u'(\sigma)^2{\mathrm d}\sigma= \int_0^r q(\sigma) w(\sigma-K)^2{\mathrm d}\sigma .\]
Let $\|w\|_\infty$ be the maximum norm of $w$ on the interval $[-K,K]$.
Then we obtain
\[ -\sn(r,k)u(r)u'(r)\le (-h)r\|w\|_\infty^2\quad\text{for $0\le r\le 2K$}. \]
Using \eqref{ineqsn} we find
\[  -u(r)u'(r)\le (-h) K\|w\|_\infty^2\quad\text{for $0\le r\le K$}. \]
Upon integrating the last inequality from $r$ to $K$ we get
\begin{equation}\label{2:est3}
 u(r)^2-u(K)^2\le 2(-h) K^2 \|w\|_\infty^2 \quad\text{for $0\le r\le K$}.
\end{equation}
Since $u(K)=w(0)$ we obtain
\[ u(r)^2\le (1+2(-h)K^2)\|w\|_\infty.\]
Now Lemma \ref{2:l2} and \eqref{2:est9} give \eqref{2:est10}.
\end{proof}

We will need the following lemma.
\begin{lemma}\label{4:l4}
Let $u:[0,b]\to \R$ be a solution of the differential equation
\[ u''(t)=q(t)u(t),\quad \text{$t\in[0,b]$}, \]
determined by the initial conditions $u(0)=1, u'(0)=0$ or $u(0)=0, u'(0)=~1$, where $q:[0,b]\to\R$ is a continuous function.
Suppose that $q(t)\ge0$ on $[0,b]$ and $q(t)\ge \lambda^2$ on $[c,b]$ for some $\lambda>0$ and $c\in[0,b)$.
Then $u(t)>0$ and $u'(t)\ge 0$ for $t\in(0,b]$, and
\[ \frac{u(t)}{u(b)}\le 2\expe^{-\lambda(b-c)}\quad\text{for all $t\in[0,c]$}.\]
\end{lemma}
\begin{proof}
For the proof see the proof of \cite[Lemma 4.5]{BiCohlVolkmerA}.
\end{proof}

So far we considered the Lam\'e-Wangerin function $W_\nu^n(t,k)$ for $t\in(-K,K)$. In the following
we will also need this function for purely imaginary $t$. More generally, by analytic continuation, we define
$W_\nu^n(t,k)$ on the strip $|\Re t|<K$.

\begin{lemma}\label{2:l5}
Let $0<c<b<2K'$. Then there is a constant $p\in(0,1)$ independent of $\nu, n$ and $s$ such that
\[ 0\le \frac{W_\nu^n(is,k)}{W_\nu^n(ib,k)}\le 2 p^{n+\nu+1}\quad\text{for  $\nu\ge -\tfrac12$, $n\in \N_0$, $s\in[0,c]$.}\]
\end{lemma}
\begin{proof}
Let $E(s)$ be the solution of the differential equation
\[
\frac{{\mathrm d}^2E}{{\mathrm d}s^2}=q(s)E,\quad  q(s)=\Lambda_\nu^n(k)-\nu(\nu+1)+\nu(\nu+1)k'^2\sn^2(s,k'),
\]
satisfying the initial conditions
$E(0)=1, E'(0)=0$ if $n$ is even and $E(0)=0, E'(0)=1$ if $n$ is odd. Then $E(s)$ is a constant multiple of
$W_\nu^n(is,k)$.
If $\nu\ge 0$ then \eqref{2:lowerbound} gives
\[ q(s)\ge \omega^2(n+1)^2+\nu(\nu+1)k'^2\sn^2(s,k')> 0. \]
If $-\frac12\le\nu<0$ then Lemma \ref{2:l2} and $\omega>k'$ imply
\[ q(s)\ge \omega^2(n+\nu+1)^2+\nu(\nu+1)k'^2\sn^2(s,k')
\ge \tfrac14\omega^2-\tfrac14k'^2>0. \]
In both cases we apply Lemma \ref{4:l4} to obtain the desired result.
\end{proof}

\section{Peanut Harmonics in Flat-Ring Coordinates}\label{sec7}

In a previous paper \cite{BiCohlVolkmerA},
we studied the internal and external harmonics in flat-ring coordinates associated with coordinate surfaces which are flat-rings. In the present work we now investigate the harmonics associated with the coordinate surfaces which are orthogonal to the flat-rings, namely what we refer to as a peanut cyclidic surface.

\subsection{Flat-Ring Coordinates}

Flat-ring cyclide coordinates is an orthogonal curvilinear coordinate system in $\R^3$. These coordinates
are  connected to Cartesian coordinates $\r=(x,y,z)$ by the transformation
\begin{equation}\label{3:flatring}
x=R\cos\phi,\quad y=R\sin\phi,\quad z=-ik R\,\sn(s,k)\sn(it,k),
\end{equation}
where $s\in(0,2K)$, $t\in(-K',K')$, $\phi\in[-\pi,\pi)$, and
\begin{equation}\label{3:R}
 \frac1R=\frac{1}{k'}\,\dn(s,k)\dn(it,k)+\frac{k}{k'}\,\cn(s,k)\cn(it,k).
\end{equation}
If one uses Jacobi's imaginary transformation \cite[Table 22.6.1]{NIST:DLMF}, namely
\begin{equation}
\sn(iz,k)=i\ssc(z,k'),\quad
\cn(iz,k)=\nc(z,k'),\quad
\dn(iz,k)=\dc(z,k'),
\end{equation}
then we can rewrite $R$, $z$ as follows
\begin{equation}
R=\frac{k'\cn(t,k')}{k\cn(s,k)+\dn(s,k)\dn(t,k')}
\label{Rzlab},
\quad
z=\frac{kk'\sn(s,k)\sn(t,k')}{k\cn(s,k)+\dn(s,k)\dn(t,k')}.
\end{equation}

Since the Jacobian elliptic functions $\cn(s,k)$, $\sn(s,k)$, $\dn(s,k)$ depend on the modulus $k\in(0,1)$, this is actually a family of coordinate systems depending on the parameter $k\in(0,1)$.
If we set $\phi=0$, we obtain a coordinate system $s,t$ in the half-plane $x>0$, $z\in\R$.
The rectangle $s\in(0,2K)$, $t\in(-K',K')$ is mapped bijectively
onto the set
\begin{equation}\label{3:Q2}
 Q_2=\{(x,z): x>0\}\setminus
 \{(x,0)
 :x\in[0,b]\cup[b^{-1},\infty)\},
\end{equation}
where
\begin{equation}\label{3:b}
 b=\frac{1-k}{k'}=\sqrt{\frac{1-k}{1+k}}\in(0,1),
\end{equation}
for $k\in(0,1)$.
The region $Q_2$ is the right-hand half-plane with cuts along the $x$-axis from $x=0$ to $x=b$ and from $b^{-1}$ to $\infty$.
This was shown in \cite[Section 2]{BiCohlVolkmerA}.
Some coordinate lines of this coordinate system are depicted in Figure~\ref{2:fig1}.
\begin{figure}[ht]
\begin{center}
\includegraphics[clip=true,trim={0 0 0 0},width=8cm]{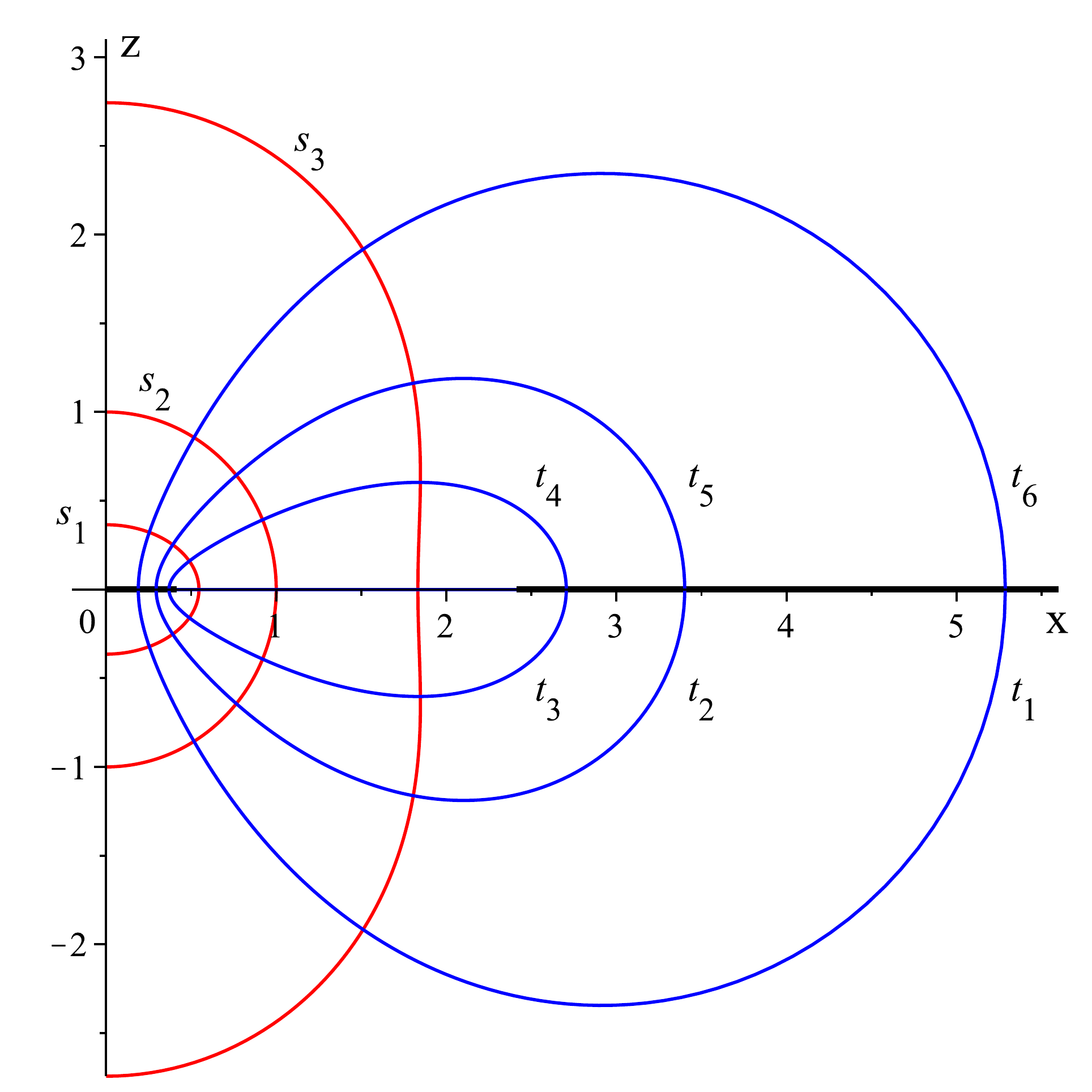}
\caption{Coordinate lines $s=s_i$, $s_1=0.5K$, $s_2=K$, $s_3=1.5K$ and $t=t_i$, $t_1=-0.7K'$, $t_2=-0.5K'$, $t_3=-0.3K'$, $t_4=0.3K'$, $t_5=0.5K'$, $t_6=0.7K'$ of planar flat-ring coordinates for $k=2^{-1/2}$.\label{2:fig1}
}
\end{center}
\end{figure}

Let $\r\ne \0$ be a point in $\R^3$. The inversion $\sigma(\r)$ of $\r$ at the unit sphere $x^2+y^2+z^2=1$ is given by
\[ \sigma(\r)=\|\r\|^{-2}\r.\]
Let $s,t,\phi$ be flat-ring coordinates of $\r$ and set $\tilde s=2K-s$.
Let $\tilde R$ be given by \eqref{3:R} with $s$ replaced by $\tilde s$.
Since $\sn(2K-s,k)=\sn(s,k)$, $\cn(2K-s,k)=-\cn(s,k)$, $\dn(2K-s,k)=\dn(s,k)$, a computation shows that
$R=\tilde R\, \|\r\|^2$
and this implies that the point $\tilde \r$ with flat-ring coordinates $\tilde s, t, \phi$ agrees with $\sigma(\r)$.
Therefore, the inversion at the unit sphere is given in flat-ring coordinates by $s\mapsto 2K-s$.

For a fixed value $s_0\in(0,2K)$
the coordinate surface $s=s_0$ describes a closed surface (adding two intersection points with the $z$-axis where the coordinate system is not valid). We call this closed surface a peanut (see Figure~\ref{3:peanut3d}).
The surface $s=2K-s_0$ is the inversion of the surface $s=s_0$. Note that the surface $s=K$ is the unit sphere.
Let $D_2$ denote the interior of the peanut surface $s=s_0$ which is given by $s<s_0$ (adding parts of the $z$-axis and the
 disk $R\le b$ in the plane $z=0$).

\begin{figure}[ht]
\centering
    \includegraphics[clip=true,trim={0 15cm 2cm 4cm},width=15cm]{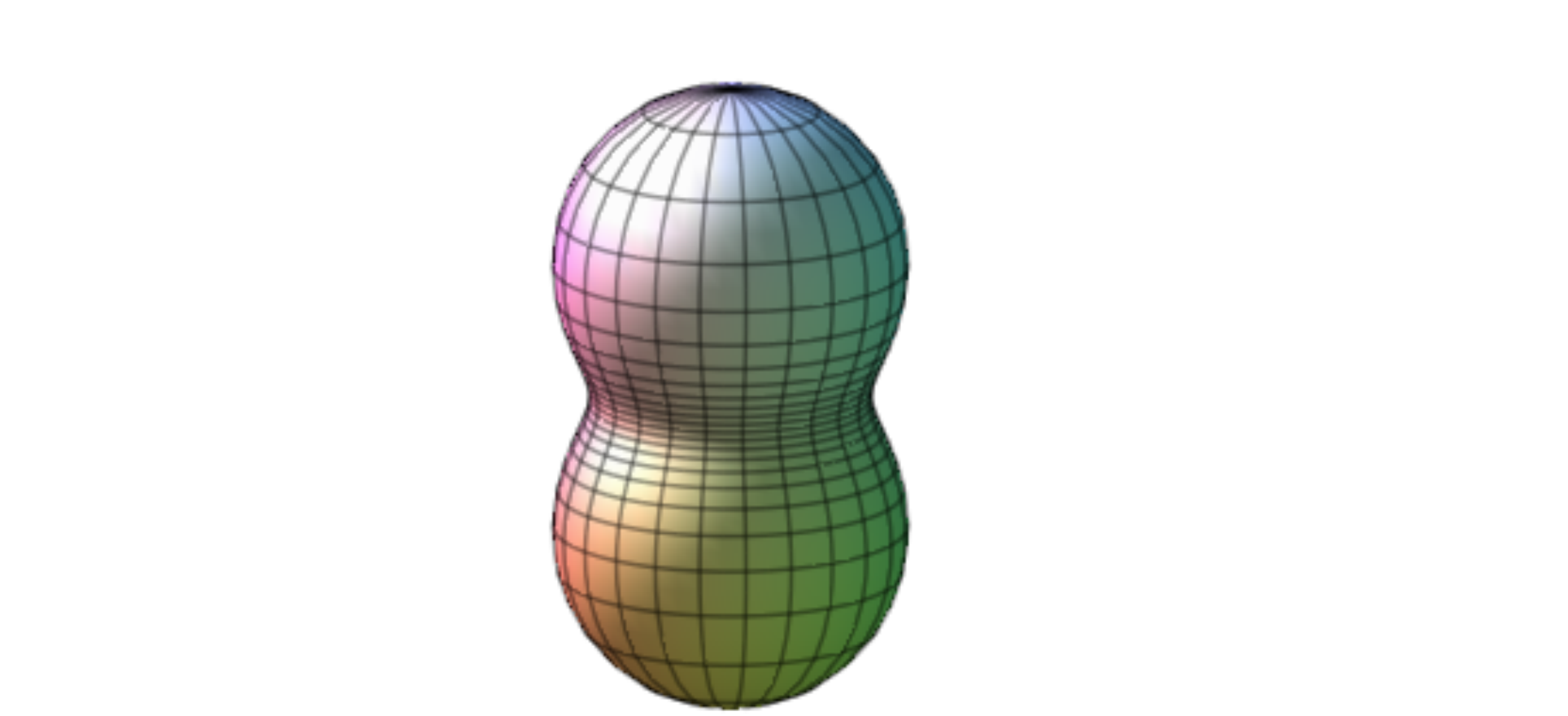}
    \caption{Peanut, $k=0.5$, $s=1.7 K$. \label{3:peanut3d}}
\end{figure}

The surface $s=s_0$ as well as the domain $D_2$ can be expressed in cartesian coordinates as follows.
If $0<s_0<2K$, $s_0\ne K$,  we define the function
\begin{equation}\label{3:Omega}
 \Omega(\r):=\frac{k^2(\|\r\|^2+1)^2}{\dn^2(s_0,k)} - \frac{(\|\r\|^2-1)^2}{\cn^2(s_0,k)}+ \frac{4z^2}{\sn^2(s_0,k)}.
\end{equation}
If $0<s_0<K$ then the coordinate surface $s=s_0$ is given by
the part of the surface $\Omega(\r)=0$ \cite[(2.11)]{BiCohlVolkmerA}
which lies inside the unit sphere.
If $K<s_0<2K$ then the coordinate surface $s=s_0$ is given by the part of the surface $\Omega(\r)=0$ which lies outside of
the unit sphere.
All coordinate surfaces $s=s_0$ intersect the plane $z=0$ in the annulus $b<R<b^{-1}$, where $R=\sqrt{x^2+y^2}$.
Between the surfaces $s=s_0$ and $s=2K-s_0$ we have $\Omega(\r)>0$ and interior to the smaller surface and exterior to the larger surface we have $\Omega(\r)<0$. We can check this by considering $\r=\0$ and $\|\r\|=1$.

We distinguish three cases:
\begin{itemize}
\item
If $s_0=K$ then the surface $s=s_0$ is the unit sphere and $D_2$ is the open unit ball
\[ B=\{ \r: \|\r\|<1\}.\]
Notice that $\Omega(\r)$ is not well-defined in this case since \cite[Table 22.5.1]{NIST:DLMF}
$\cn(K,k)=0$.
\item
If $0<s_0<K$ then the surface $s=s_0$ lies in $B$ and
\[ D_2=B \cap\{\r:\Omega(\r)<0\}.\]
\item
If $K<s_0<2K$ then the surface $s=s_0$ is exterior to the unit sphere and
\[ D_2=B \cup\{\r:\Omega(\r)>0\}.\]
\end{itemize}

\subsection{Internal Peanut Harmonics}

Let $m\in\Z$, $\nu=|m|-\frac12$, $k\in(0,1)$, $h\in\C$ and $\lambda=\nu(\nu+1)-h$.
Let $u_1(s)$, $0<s<2K$, be a solution of the Lam\'e equation \eqref{2:lame},
and let $u_2(t)$, $-K'<t<K'$, be a solution of
the modified Lam\'e equation \eqref{2:modlame0}.
Let $\r=(x,y,z)$, $R=\sqrt{x^2+y^2}$ \eqref{Rzlab}.
Then, by \cite[Theorem 3.1]{BiCohlVolkmerA},
the function
\begin{equation}\label{3:sepsol}
 u(\r)=\frac{1}{\sqrt{R}}
u_1(s)u_2(t)\expe^{im\phi},
\end{equation}
is harmonic in $\R^3$ except for the $z$-axis and the set $\{(x,y,0): R\le b \text{ or } R \ge b^{-1}\}$,
where 
\begin{equation}
b:=b(k)=\sqrt{\frac{1-k}{1+k}},
\label{bdef}
\end{equation}
is defined in \eqref{3:b}.
Such a function will be called an internal peanut harmonic if it is harmonic on each domain $D_2$ considered in
\S3.1.
Therefore, an internal peanut harmonic is harmonic on the union of all domains $D_2$, that is, on all of $\R^3$ except for the set
$\{(x,y,0): R\ge b^{-1}\}$.
In this case $u(\r)$ has to stay bounded when we approach the positive and negative $z$-axis, so the function
$|t\mp K'|^{-1/2}u_2(t)$ has to stay bounded as $t\to\pm K'$.
Therefore, we arrive at the eigenvalue problem treated in Section \ref{sec2}. Correspondingly, we take $\lambda=\Lambda^n_{|m|-\frac12}(k')$ and $u_2(t)=W_{{|m|-\frac{1}{2}}}^n(t,k')$
Then we require that the function $u_1(s)u_2(t)$ is analytic in the right-hand half plane $x>0, z\in \R$ except the segment between $b^{-1}$ and $+\infty$ on the $x$-axis.
As in \cite[Section 5.1]{BiCohlVolkmerA} we see that this implies that $u_1$ and $u_2$ are both even or both odd functions.
We take $u_1(s)=u_2(is)$.
Thus we define internal peanut harmonics by
\begin{equation}\label{3:intpeanut}
G^n_m(\r)=
\frac{1}{\sqrt{R}}
W^n_{|m|-\frac12}(is,k')W^n_{|m|-\frac12}(t,k')\expe^{im\phi}\quad\text{for $m\in\Z$, $n\in\N_0$}.
\end{equation}

We collect some properties of internal peanut harmonics in the following theorem.
\begin{thm}\label{3:internal}
The internal peanut harmonics $G_m^n$ are harmonic on $\R^3$ except for the set $\{(x,y,0): R\ge b^{-1}\}$.
Moreover, we have
\begin{equation}\label{3:Greflection}
 G^n_m(x,y,-z)=(-1)^{n} G^n_m(\r).
\end{equation}
\end{thm}
\begin{proof}
From our discussion at the beginning of this subsection we know that $G^n_m$ is a harmonic function on $\R^3$ except in the
annulus $\{(x,y,0): x^2+y^2\ge b^{-2}\}$, the $z$-axis, and the circle centered at the origin with radius $b$ in the $xy$-plane $G^n_m$ is bounded in a neighborhood of the circle. Therefore, the circle is a removable singularity of $G^n_m$ \cite[Theorem XIII, page 271]{Kellogg}. Since $G^n_m$ stays bounded when we approach the $z$-axis, the $z$-axis is also a removable singularity.
Hence, $G^n_m$ is harmonic on the desired domain.
The reflection $z\mapsto -z$ is expressed by $t\mapsto-t$ which implies \eqref{3:Greflection}{.}
\end{proof}

\subsection{The Dirichlet Problem {for Internal Peanut Harmonics}}

Theorem \ref{2:t2} implies the following theorem.

\begin{thm}\label{3:t1}
The system of functions
\[
J_{m,n}^k(t,\phi):=\frac{1}{\sqrt{2\pi}} W^{\,n}_{|m|-\frac12}(t,k')\expe^{im\phi},\quad m\in\Z, n\in\N_0 \]
is an orthonormal basis in the Hilbert space
\[ H_2=L^2((-K',K')\times (-\pi,\pi)). \]
\end{thm}
We now solve the Dirichlet problem for the peanut region $D_2$ given by  $s<s_0$.
We say that a harmonic function $u$ defined in $D_2$ attains
the boundary values $f$ on $\partial D_2$
in the weak sense if
$\sqrt{R}u$
(expressed in terms of flat-ring coordinates $s,t,\phi$) evaluated at $s_1\in(0,s_0)$ converges to 
$\sqrt{R}f$
in the Hilbert space $H_2$ as $s_1\to s_0$. Notice that the peanut region $D_2$ (in contrast to the flat-ring region $D_1$) meets the $z$-axis so that the factor
$\sqrt{R}$
cannot be omitted in this definition. As in \cite[Section 5.2]{BiCohlVolkmerA},
the solution of the Dirichlet problem is unique.

\begin{thm}\label{3:Dirichlet}
Let $f$ be a function defined on the boundary $\partial D_2$ of the region $D_2$ for some $s_0\in(0,2K)$.
Suppose that $f$ is represented in flat-ring coordinates as
\[ \sqrt{R}f(\r)=g(t,\phi) ,\quad  t\in(-K',K'),\quad \phi\in(-\pi,\pi)\]
such that $g\in H_2$.
For all $m\in\Z$ and $n\in\N_0$. Define
\begin{eqnarray*}
&& \hspace{-3.0cm}c_m^{n}:=\frac{1}{2\pi W^{n}_{|m|-\frac12}(is_0,k')}\int_{-\pi}^\pi
 \expe^{-im\phi}
 \int_{-K'}^{K'} g(t,\phi)W^{n}_{|m|-\frac12}(t,k')\,{\mathrm d}t\, {\mathrm d}\phi\\
 &&\hspace{-2.33cm}=  \frac{1}{2\pi\Bigl\{W^{n}_{|m|-\frac12}(is_0,k')\Bigr\}^2}\int_{\partial D_2} \frac{1}{h(\r)} f(\r) G_{-m}^{n}(\r){\mathrm d}S(\r),
\end{eqnarray*}
where $h(\r)= kR\left(\sn^2(s,k)-\sn^2(it,k)\right)^{1/2}$.
Then the function
\begin{equation}\label{3:sol}
 u(\r)=\sum_{m\in\Z}\sum_{n=0}^\infty c_m^{n}  G_m^{n}(\r)
\end{equation}
is harmonic in $D_2$ and it attains the boundary values $f$ on $\partial D_2$ in the weak sense.
The infinite series in \eqref{3:sol} converges absolutely and uniformly in compact subsets of $D_2$.
\end{thm}
\begin{proof}
As in the proof of \cite[Theorem 5.3]{BiCohlVolkmerA},
we see that the two formulas for $c_m^n$ agree.
Using \eqref{3:R}
we find
\[ 
\frac{1}{R}=\frac1{k'}\dn(s,k)\dc(t,k')+\frac{k}{k'} \cn(s,k)\nc(t,k')
\le \frac{2}{k'} \dc(t,k') .\]
Let $t\in(-K',K')$, $\phi\in(-\pi,\pi)$  and $0<s\le s_1<s_0$. Using Lemmas \ref{2:l3}, \ref{2:l4}, \ref{2:l5}
we estimate
\[
\frac{1}{\sqrt{R}}\left|\frac{W_{|m|-\frac12}^n(is,k')}{W_{|m|-\frac12}^n(is_0,k')}
W_{|m|-\frac12}^n(t,k')\expe^{im\phi}\right|\le
C p^{|m|+n} (1+|m|+n) ,\]
where the constants $C$ and $p\in(0,1)$ are independent of $m,n,s,t,\phi$.
Since $\{c_m^n\}$ is a bounded double sequences, this proves that the
series in \eqref{3:sol} is absolutely and uniformly convergent
on compact subsets of $D_2$. Consequently, by Theorem \ref{3:internal}, $u$ defined by \eqref{3:sol} is harmonic
in $D_2$.
We show that  $u$ attains the boundary values $f$ on $\partial D_2$ in the weak sense
by the same method as used in the proof of \cite[Theorem 5.3]{BiCohlVolkmerA}. In this argument we use
Theorem \ref{3:t1}.
\end{proof}

\subsection{External Peanut Harmonics}
External peanut harmonics are harmonic functions $u$ of the form \eqref{3:sepsol} which are harmonic outside
any peanut region $D_2$, that is, on all of $\R^3$ except the disk centered at the origin with radius $b$ in the $xy$-plane.
External peanut harmonics can simply be defined by
the Kelvin transformation \cite[Ch.\ IX, \S 2]{Kellogg}
of internal peanut harmonics. Note that this method was not available for flat-ring harmonics because
the Kelvin transformation of an internal flat-ring harmonic is again an internal flat-ring harmonic.

More explicitly, we define external peanut harmonics by
\begin{equation}\label{3:extpeanut}
 H^n_m(\r)=\frac{1}{\sqrt{R}} W^n_{|m|-\frac12}(2iK-is,k') W^n_{|m|-\frac12}(t,k')\expe^{im\phi}\quad\text{for $m\in\Z$, $n\in\N_0$.}
\end{equation}

\begin{thm}\label{3:external}
External peanut harmonics $H_m^n$ are harmonic functions on $\R^3$ except for the disk centered at the origin with radius $b$ in the $xy$-plane. Moreover,
\begin{eqnarray}
&&\hspace{-8cm} G^n_m(\sigma(\r))=||\r|| H^n_m(\r),\label{3:ext2}\\
&&\hspace{-8cm}H^n_m(x,y,-z)=(-1)^n H^n_m(\r), \label{3:ext3}\\
&&\hspace{-8cm}\lim_{\|\r\|\to\infty} H^n_m(\r)=0. \label{3:ext4}
\end{eqnarray}
\end{thm}

\begin{proof}
As mentioned in Section 3.1 inversion at the unit sphere is expressed in flat-ring coordinates by
$s\mapsto 2K- s$.
This implies \eqref{3:ext2}.
Theorem \ref{3:internal} and \eqref{3:ext2} shows that $H_m^n$ is harmonic on the desired domain.
Equations \eqref{3:ext3} and \eqref{3:ext4} also follow from Theorem \ref{3:internal} and \eqref{3:ext2}.
\end{proof}

External harmonics admit an integral representation in terms of internal harmonics. Define the Wronskian $w_m^n$ by
\begin{equation}\label{3:wronskian}
w_m^n:=V(s)U'(s)-V'(s)U(s),
\end{equation}
where $U(s):=W^n_{|m|-\frac12}(is,k')$, $V(s):=U(2K-s)$.
The function $U(s)$ is a constant multiple of the solution $E(s)$ of \eqref{2:lame}
determined by initial conditions
$E(0)=1$, $E'(0)=0$ if $n$ is even and $E(0)=0$, $E'(0)=1$ if $n$ is odd. This equation has the form $E''=q(s)E$ with $q(s)>0$. Therefore, $E(s)>0$ and $E'(s)>0$ for $s>0$ which implies that $w_m^n\ne 0$.

\begin{thm}\label{3:intrep}
Let $s_0\in(0,2K)$, $m\in\Z$, $n\in\N_0$, and let $\r^\ast$ be  a point outside $\overline{D}_2$, where $D_2$ is the region given by $s<s_0$. Then
\begin{equation}\label{3:intrep1}
 H_m^n(\r^\ast)=\frac{w_m^n}{4\pi\{W_{|m|-\frac12}^n(is_0,k')\}^2}\int_{\partial D_2} \frac{G_m^n(\r)}{h(\r)\|\r-\r^\ast\|}{\mathrm d}S(\r).
\end{equation}
\end{thm}

We omit the proof of this theorem which is very similar to the proof of
\cite[Theorem 5.5]{BiCohlVolkmerA}.

\subsection{Expansion of the Fundamental Solution}
We obtain the desired expansion of $\|\r-\r^\ast\|^{-1}$ in internal and external peanut harmonics by combining Theorems  \ref{3:Dirichlet} and \ref{3:intrep}.

\begin{thm}\label{3:expansion}
Let $\r,\r^\ast\in \R^3$ with flat-ring coordinates $s,s^\ast\in(0,2K)$, $t,t^\ast\in(-K',K')$, respectively.
If $s<s^\ast$ then
\begin{equation}\label{3:expansion2}
\frac{1}{\|\r-\r^\ast\|}=2\sum_{m\in\Z}\sum_{n=0}^\infty
\frac{1}{w_m^n} G_m^n(\r)H_{-m}^n(\r^\ast).
\end{equation}
\end{thm}

\noindent Since we now have
\eqref{3:expansion2}, we can follow \cite{Cohlerratum12,CTRS} in order
to obtain an addition theorem for the associated
Legendre function of the second kind with odd-half-integer degree in terms of Lam\'e-Wangerin functions.
This proceeds  
through comparison of 
the azimuthal Fourier component of the $1/r$ potential in rotationally-invariant 
coordinate systems which separates Laplace's equation, such as flat-ring cyclide coordinates. 
\begin{thm}
\label{addnthm}
Let {$m\in\N_0$}, $0<s<s^\ast<2K$, $t,t^\ast\in(-K',K')$.
Then
\begin{eqnarray}
&&\hspace{-1.0cm}
Q_{m-\frac12}(\chi)
=
2\pi
\sum_{n=0}^\infty
\frac{1}{w_m^n}
W^n_{{m}-\frac12}(is,k')
W^n_{{m}-\frac12}(t,k')
W^n_{{m}-\frac12}(2iK-is^\ast,k')
W^n_{{m}-\frac12}(t^\ast,k'),
\end{eqnarray}
where $\chi:((0,2K)\times(-K',K'))^2\times(0,1)\to\R$ is given by 
\begin{eqnarray}
&&\hspace{-0.85cm}\chi
:=\chi(s,t,s^\ast,t^\ast;k)
\nonumber\\
&&\hspace{-0.50cm}:=k^2\sn(s,k)\sn(it,k)\sn(s^\ast,k)\sn(it^\ast,k)
 -\frac{k^2}{k'^2}\cn(s,k)\cn(it,k)\cn(s^\ast,k)\cn(it^\ast,k)\nonumber\\
 &&\hspace{6.05cm}+\frac{1}{k'^2}\dn(s,k)\dn(it,k)\dn(s^\ast,k)\dn(it^\ast,k).
 \label{chist}
\end{eqnarray}

\end{thm}
\begin{proof}
This follows from comparison of \eqref{3:expansion2} with the azimuthal
Fourier expansion \cite[(15)]{CT}
\[ \frac{1}{\|\r-\r^\ast\|}=\frac{1}{\pi\sqrt{RR^\ast}}\sum_{m=0}^\infty Q_{m-\frac12}(\chi)\expe^{im(\phi-\phi^\ast)},\]
where
\[
\chi=\frac{R^2+{R^\ast}^2+(z-z^\ast)^2}{2RR^\ast},
\]
with $R$, $R^\ast$, $z$, $z^\ast$ given
in terms of $s,t$ and $s^\ast, t^\ast$
respectively \eqref{Rzlab},
and \eqref{chist} is given in \cite[Lemma 5.7]{BiCohlVolkmerA}.
\end{proof}
{
The addition Theorem \ref{addnthm} leads to an integral relation for Lam\'e-Wangerin functions.
}
{
\begin{thm}\label{3:intrel1}
Let $m,n\in\N_0$, $0<s<s^\ast<2K$, $-K'<t^\ast<K'$.
Then
\begin{equation} \int_{-K'}^{K'} Q_{m-\frac12}(\chi) W_{m-\frac12}^n(t,k')\,{\mathrm d}t= \frac{2\pi}{w_m^n} W^n_{m-\frac12}(is,k')
W^n_{m-\frac12}(2iK-is^\ast,k')
W^n_{m-\frac12}(t^\ast,k').
\end{equation}

\end{thm}
}


\noindent
Theorem \ref{3:intrel1} is a new result.
However, we are able to improve upon it by using the method of fundamental solutions employed in \cite{Volkmer84}.

In order to simplify notation we set
\begin{equation}\label{3:V}
V_\nu(s):=W_\nu^n(is,k)\quad\text{for $|\Im s|<K'$}
\end{equation}
for some $n\in\N_0$, $0<k<1$, $\nu\ge -\frac12$.
Then $V_\nu$ is a solution of Lam\'e's equation \eqref{2:lame} with $h=\nu(\nu+1)-\Lambda_\nu^n(k)$,
and it is a Fuchs-Frobenius solution at both regular singular points $s=\pm iK'$ belonging to the exponent $\nu+1$.

\begin{thm}\label{3:intrel2}
Let $\LW_\nu(z)$ be as in \eqref{3:V}.
If $0<s_0<2K$, $-s_0<s_1<s_0$, $-K'<t_0<K'$ then
\begin{equation}\label{inteq1}
2\pi\, \LW_\nu(2K-s_0)\LW_\nu(it_0)\LW_\nu(s_1)=[\tilde \LW_\nu,\LW_\nu]\int_{-K'}^{K'} Q_\nu(\chi(s_1,t,s_0,t_0))\LW_\nu(it)\,{\mathrm d}t,
\end{equation}
where $[\LW_\nu,\tilde \LW_\nu]$ denotes the Wronskian of $\LW_\nu$ and $\tilde \LW_\nu(z)=\LW_\nu(2K-z)$.
Since $\LW_\nu$ and $\tilde \LW_\nu$ are both solutions of \eqref{2:lame}, the Wronskian is a constant.
\end{thm}

\begin{figure}
\centering
\setlength{\unitlength}{1cm}
\thicklines
\begin{picture}(6,6)(-1,-3)
\put(-1,0){\vector(1,0){6}}
\put(-1,2){\line(1,0){6}}
\put(-1,-2){\line(1,0){6}}
\put(-0.5,-3){\vector(0,1){6}}
\put(0,-1.5){\vector(1,0){2}}
\put(2,-1.5){\line(1,0){2}}
\put(4,-1.5){\vector(0,1){2.3}}
\put(4,0){\line(0,1){1.5}}
\put(4,1.5){\vector(-1,0){2}}
\put(0,1.5){\line(1,0){2}}
\put(0,1.5){\vector(0,-1){2.3}}
\put(0,0){\line(0,-1){1.5}}
\put(-1,2.1){$K'$}
\put(-1.3,-2.4){$-K'$}
\put(5.05,0.1){$s$}
\put(-0.4,2.9){$t$}
\put(-0.9,1.4){$t_2$}
\put(-0.9,-1.5){$t_1$}
\put(0.1,0.2){$s_1$}
\put(4.1,0.15){$2K$}
\put(1.8,-1.3){$C_1$}
\put(1.8,1.1){$C_3$}
\put(0.2,-0.8){$C_4$}
\put(3.4,0.6){$C_2$}
\put(1,0.5){\circle*{0.1}}
\put(1.1,0.5){$(s_0,t_0)$}
\end{picture}
\caption{Path of integration used in \eqref{Hansint}. \label{fig1}}
\end{figure}
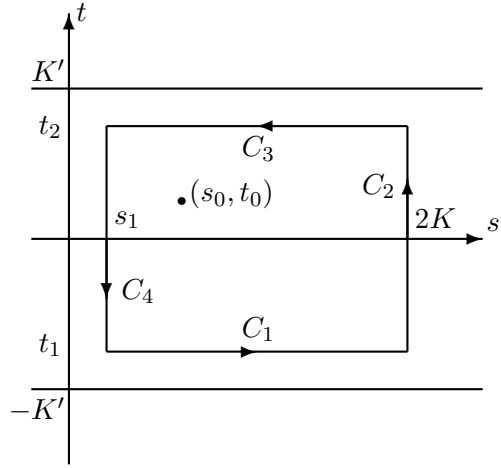

\begin{proof}
We define
\begin{eqnarray*}
&&\hspace{-9.3cm}u(s,t)= \LW_\nu(2K-s)\LW_\nu(it),\\
&&\hspace{-9.3cm}v(s,t)=Q_\nu(\chi(s,t,s_0,t_0)).
\end{eqnarray*}
The function $v(s,t)$ is well-defined for $(s,t)\in\R\times (-K',K')$ except for logarithmic singularities at
the points $(s_0+4 jK,t_0)$ and $(-s_0+4 j K,-t_0)$ with $j\in\Z$ \cite[Lemma 1.3]{Volkmer84}.
Let $s_2=2K$ and $-K'<t_1<t_0<t_2<K'$.
Let $C_1+C_2+C_3+C_4$ be the rectangular path as shown in Figure \ref{fig1}. This path forms the boundary of the
rectangle $[s_1,s_2]\times [t_1,t_2]$.
This rectangle contains the point $(s_0,t_0)$ but none of the other singularities of $v$. According to \cite[Theorem 1.11]{Volkmer84} we have
\begin{equation}
2\pi\, u(s_0,t_0)=\sum_{j=1}^4 \int_{C_j}\bigl( (u\partial_2v-v\partial_2u)\,{\mathrm d}s+(v\partial_1u-u\partial_1 v)\,{\mathrm d}t\bigr).
\label{Hansint}
\end{equation}
By our assumption on $\LW_\nu$ and \cite[Lemma 2.6]{Volkmer84}, the integral $\int_{C_1} (u\partial_2v-v\partial_2u)\,{\mathrm d}s$ converges to $0$ as $t_1\to -K'$, and
the integral $\int_{C_3} (u\partial_2v-v\partial_2u)\,{\mathrm d}s$ converges to $0$ as $t_2\to K'$.
Therefore, one obtains
\begin{equation}\label{eq1}
 2\pi \LW_\nu(2K-s_0)\LW_\nu(it_0)=I_1+I_2+I_3+I_4,
\end{equation}
where
\begin{eqnarray*}
&&\hspace{-7cm}I_1= \LW_\nu'(2K-s_1)\int_{-K'}^{K'} v(s_1,t)\LW_\nu(it)\,{\mathrm d}t,\\
&&\hspace{-7cm}I_2= \LW_\nu(2K-s_1)\int_{-K'}^{K'} \partial_1v(s_1,t)\LW_\nu(it)\,{\mathrm d}t,\\
&&\hspace{-7cm}I_3= -\LW_\nu'(0)\int_{-K'}^{K'} v(2K,t)\LW_\nu(it)\,{\mathrm d}t,\\
&&\hspace{-7cm}I_4= -\LW_\nu(0)\int_{-K'}^{K'} \partial_1v(2K,t)\LW_\nu(it)\,{\mathrm d}t.
\end{eqnarray*}
The function $\LW_\nu$ is even or odd. If $\LW_\nu$ is even then $\LW_\nu'(0)=0$ so $I_3=0$. If $\LW_\nu$ is odd then $v(2K,t)\LW_\nu(it)$ is an
odd function of $t$, so again $I_3=0$. In a similar way, we see that $I_4=0$.

\medskip
We now apply \cite[Theorem 1.11]{Volkmer84} to the
 counter-clockwise rectangular path $\tilde C$ defined using the vertices $(s_1,t_1)$, $(s_1,t_2)$,
$(0,t_1)$, $(0,t_2)$. This time we replace $u$ by $\tilde u(s,t)=\LW_\nu(s)\LW_\nu(it)$.
The path $\tilde C$ does not wind around a singularity of $v(s,t)$ so
\[ \int_{\tilde C} (\tilde u\partial_2v-v\partial_2\tilde u)\,{\mathrm d}s+(v\partial_1\tilde u-\tilde u\partial_1 v)\,{\mathrm d}t=0.
\]
As before, we let $t_1\to-K'$, $t_2\to K'$, and  note that the line integral along the segment
from $(0,-K')$ to $(0,K')$ vanishes.
Therefore, we obtain
\begin{equation}\label{eq2}
 \LW_\nu'(s_1)\int_{-K'}^{K'} v(s_1,t)\LW_\nu(it)\,{\mathrm d}t=\LW_\nu(s_1)\int_{-K'}^{K'} \partial_1v(s_1,t) \LW_\nu(it)\,{\mathrm d}t.
\end{equation}
If we combine \eqref{eq1} with $I_3=I_4=0$ and \eqref{eq2} we obtain \eqref{inteq1}.
\end{proof}

Theorem \ref{3:intrel2} implies Theorem \ref{3:intrel1} when we set $s_0=s^\ast$, $s_1=s$, $t_0=t^\ast$, $\nu=m-\frac12$
and replace $k$ by $k'$.
In \cite[pages 79, 80]{ErdelyiHTFIII} integral equations for Lam\'e-Wangerin functions are mentioned. An improved version of these
integral equations
can be obtained from \eqref{inteq1} by a limiting process as shown in the following theorem.

\begin{thm}\label{3:intrel3}
Let $\nu\ge -\frac12$ and $\LW_\nu(s)$ be as in \eqref{3:V}.
If $t_0\in(-K',K')$ then
\begin{equation}\label{inteq2} \LW_\nu(it_0)=
\frac{\expe^{\frac12(\nu+1)i\pi}\Gamma(\nu+1)}{2^{\nu+2}\sqrt{\pi}\,\Gamma(\nu+\frac32)}
\frac{[\tilde \LW_\nu,\LW_\nu]}{L_\nu(k)\LW_\nu(K-iK')}
\int_{-K'}^{K'} (f(t,t_0))^{-\nu-1}\LW_\nu(it)\,{\mathrm d}t,
\end{equation}
where
\begin{eqnarray*}
&&\hspace{-5cm}L_\nu(k):= \lim_{u\to K'-} \cn(u,k')^{-\nu-1} \LW_\nu(iu),\\
&& \hspace{-5cm}f(t,t_0):=k\sn(it,k)\sn(it_0,k)+\frac{k}{k'}\cn(it,k)\cn(it_0,k).
\end{eqnarray*}
\end{thm}
\begin{proof}
Let $t,t_0,u\in(-K',K')$. Then the function $\chi$ defined in \eqref{chist} satisfies
\[\chi(i u,t,K+iu,t_0;k)= ikf(t,t_0)\ssc(u,k')\nd(u,k')+\frac1{k'}\dn(it,k)\dn(it_0,k).
\]
It follows that
\begin{equation}\label{ineqchi}
 \Re\, \chi(i u,t,K+iu,t_0;k)=\frac1{k'}\dn(it,k)\dn(it_0,k)\ge \frac1{k'}>1 .
\end{equation}
Therefore, $Q_\nu(\chi(i u,t,K+iu,t_0;k))$ is an analytic function of $(t,t_0,u)\in(-K',K')^3$.
By analytic continuation it can be shown that \eqref{inteq1} implies
\begin{equation}\label{eq3}
2\pi\, \LW_\nu(K-iu)\LW_\nu(it_0)\LW_\nu(iu)=[\tilde \LW_\nu,\LW_\nu]\int_{-K'}^{K'} Q_\nu(\chi(iu,t,K+iu,t_0;k))\LW_\nu(it)\,{\mathrm d}t.
\end{equation}
We multiply \eqref{eq3} on both sides by $\cn^{-\nu-1}(u,k')$ and take the limit as $u\to~K'$.
Then we have on the left-hand side
\[ \lim_{u\to K'} \cn^{-\nu-1}(u,k')2\pi\, \LW_\nu(K-iu)\LW_\nu(it_0)\LW_\nu(iu)=2\pi\, \LW_\nu(K-iK')\LW_\nu(it_0) L_\nu(k).\]
On the right-hand side we have
\[ [\tilde \LW,\LW] \lim_{u\to K'} \int_{-K'}^{K'} \cn^{-\nu-1}(u,k')Q_\nu(\chi(iu,t,K+iu,t_0;k))\LW(it)\,{\mathrm d}t .\]
Suppose that the limit can be taken inside the integral. So we consider
\begin{equation}\label{eq4}
 \lim_{u\to K'}  \cn^{-\nu-1}(u,k')Q_\nu(\chi(iu,t,K+iu,t_0;k)) =\expe^{-\frac12(\nu+1)i\pi}
 \frac{\sqrt\pi\,\Gamma(\nu+1)}{2^{\nu+1}\Gamma(\nu+\frac32)} \left( f(t,t_0)\right)^{-\nu-1}.
 \end{equation}
This follows from the asymptotic behavior of the Legendre function $Q_\nu$ \cite[(14.8.15)]{NIST:DLMF}
\[ Q_\nu(z)=\frac{\sqrt\pi\,\Gamma(\nu+1)}{2^{\nu+1}\Gamma(\nu+\frac32)} z^{-\nu-1}(1+O(z^{-2}))\quad\text{as $|z|\to\infty$} .\]

We now justify the interchange of limit and integral.
%
\noindent Note that \eqref{ineqchi}
implies that there is a constant $C>0$ (independent of $t,t_0,u$) such that
\begin{equation}\label{ineq1}
 |Q_\nu(\chi(iu,t,K+iu,t_0;k))|\le C |\chi(iu,t,K+iu,t_0;k)|^{-\nu-1} .
\end{equation}
Moreover, we have
\begin{equation}\label{ineq2}
 f(t,t_0)=\frac{k(k'^{-1}-\sn(t,k')\sn(t_0,k'))}{\cn(t,k')\cn(t_0,k')}\ge k\left(\frac1{k'}-1\right)>0.
\end{equation}
Now \eqref{ineq1} and \eqref{ineq2} give
\begin{eqnarray*}
\hspace{-0.8cm} \cn^{-\nu-1}(u,k')|Q_\nu(\chi(iu,t,K+iu,t_0;k))|&\le & C \left\{\cn(u,k')|\chi(iu,t,K+iu,t_0;k)|\right\}^{-\nu-1} \\
\hspace{-0.8cm}&\le & C \left\{f(t,t_0)\sn(u,k')\right\}^{-\nu-1}\\
\hspace{-0.8cm}&\le & C \left\{\frac{k}2 \left(\frac1{k'}-1\right)\right\}^{-\nu-1}
\end{eqnarray*}
provided that $\sn(u,k')\ge \frac12$.
This justifies the interchange of limit and integral by Lebesgue's bounded convergence theorem, and
therefore the proof of \eqref{inteq2}
is complete.
\end{proof}

Note that we have also verified the above addition theorem and
integral formulas numerically.

\section{Application of the $k\to 1$ limit of flat-ring coordinates}

In this section we show that, as $k\to 1$, flat-ring coordinates
becomes spherical coordinates.
Then we show how our expansion of the $1/r$ potential
in peanut harmonics becomes
the multipole expansion of the $1/r$ potential
\cite[p.~1273-1274, (10.3.37)]{MorseFesh} in spherical coordinates.

\subsection{Flat-ring coordinates in the limit $k\to1$ are spherical coordinates}

\begin{figure}[ht]
\begin{center}
\includegraphics[clip=true,trim={1.3cm 0.4cm 5.9cm 1.2cm},width=10cm]{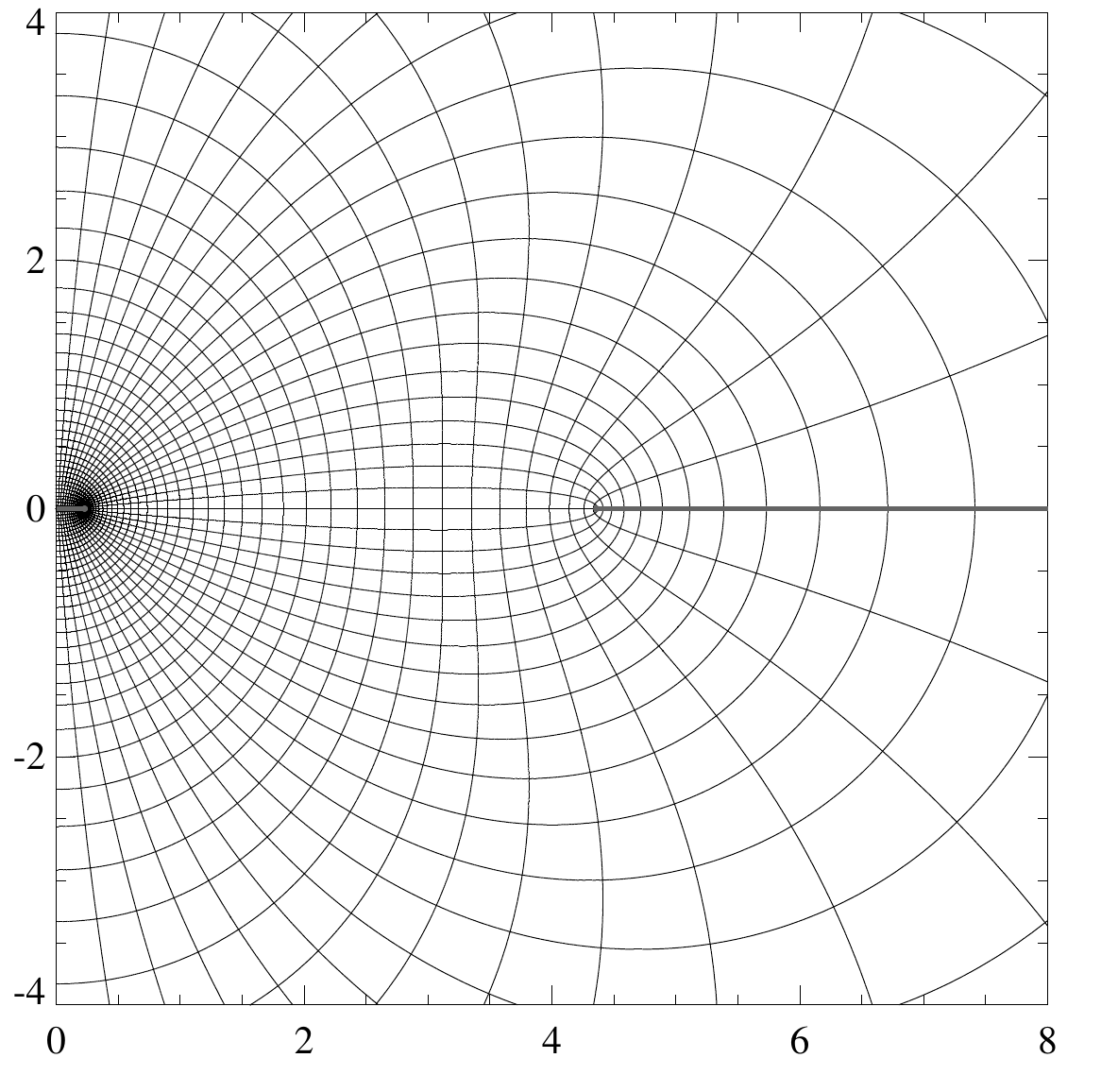}
\caption{Flat-ring cyclide coordinates with $k=\frac{9}{10}$. Then the asymptotic inner and outer radii of the flat-rings, are respectively given $b=1/\sqrt{19}\simeq 0.2294$ and $b^{-1}=\sqrt{19}\simeq 4.359$. 
Small dark grey circles are drawn
at the points $(0,b)$ and $(0,b^{-1})$
(see \eqref{bdef})
and dark grey thick line segments are drawn 
to represent the intervals $[0,b^{-1}]$ and
$[b,\infty)$.
The abscissa represents the radial coordinate $R=(x^2+y^2)^{1/2}$ and the ordinate represents the $z$-axis. One can see that as $k$ is approaching $1$, the coordinate system close to the origin is starting to resemble spherical coordinates. Similarly, as $k$ approaches $1$, the flat-ring coordinate surfaces near the plane $z=0$ have larger and larger outer-radii.\label{4:fig1}
}
\end{center}
\end{figure}

Spherical coordinates in $\R^3$, $r\ge 0$, $\theta\in[0,\pi]$, $\phi\in[-\pi,\pi)$  are connected to Cartesian coordinates $x,y,z$ by the transformation
\begin{equation}\label{1:spherical}
x=r\, \cos\phi\,\sin\theta,\quad y=r\, \sin\phi \, \sin\theta,\quad z= r\,\cos\theta.
\end{equation}
To demonstrate their connection to flat-ring coordinates, let $\sigma\in\R$ and $\tau\in(0,\pi)$. We set $s=K+\sigma$ and $t=K'-\tau$. Then $t\in(-K',K')$ and,
for $k$ sufficiently close to $1$, $s\in(0,2K)$.
Now
\begin{equation}\label{4:R}
 R=\frac{\dn(\sigma,k)\sn(\tau,k')}{1-\sn(\sigma,k)\dn(\tau,k')} .
 \end{equation}
Therefore,
\begin{equation}\label{Rlimit}
 R\to \frac{\sech\sigma\sin\tau}{1-\tanh \sigma}=\expe^{\sigma}\sin\tau\quad \text{as $k\to 1$}
 \end{equation}
and so
\[ \lim_{k\to 1} x=\expe^\sigma\sin\tau\cos \phi,\quad \lim_{k\to 1}y=\expe^\sigma\sin\tau\sin \phi .\]
Moreover,
\[ \lim_{k\to 1} z =\lim_{k\to1} \frac{\cn(\sigma,k)\cn(\tau,k')}{1-\sn(\sigma,k)\dn(\tau,k')}=\frac{\sech\sigma\cos\tau}{1-\tanh \sigma}=\expe^\sigma\cos\tau. \]
Therefore, flat-ring coordinates approach spherical coordinates $r, \theta, \phi$ in the limit $k\to 1$ (with $\theta=\tau, r=\expe^\sigma$.)

\subsection{The limit of the Lam\'e-Wangerin functions $W_\nu^n(t,k)$ as $k\to 0$}

Note that in our above expansions we have
Lam\'e-Wangerin functions which are a function
of $k'=\sqrt{1-k^2}$. So for those Lam\'{e}-Wangerin
functions, the limit as $k\to 1$ is equivalent
to the limit of the Lam\'e-Wangerin functions with argument $k$ as $k\to 0$.

\medskip
Let $w(t)=\F_\nu(t,\lambda,k)$
be the solution of equation \eqref{2:modlame} such that \eqref{2:bc1} holds with $a_0=1$.
By analytic continuation, this function is well-defined in the strip $-K<\Re t<K$.
We note that the Lam\'e-Wangerin functions $W_\nu^n$ can be written as
\begin{equation}
W_\nu^n(t,k)=d_\nu^n(k) \F_\nu(t,\Lambda_\nu^n(k),k),
\end{equation}
where the constants $d_\nu(k)>0$ are chosen such that the normalization integral
\eqref{LWnorm} is satisfied.

We recall the following lemma 
\cite[Lemma 6.3]{BiCohlVolkmerA}. 

\begin{lemma}\label{4:l3}
Let $D$ be a simply-connected domain in the complex plane $\C$ containing $0$.
Let $p_n: D\to\C$ be a sequence of analytic functions for each $n\in\N\cup\{\infty\}$ such that $p_n(z)\to p_\infty(z)$ locally uniformly
on $D$, and $p_n(0)=\nu(\nu+1)$ for all $n\in\N\cup\{\infty\}$, where $\nu\ge -\frac12$. For each $n\in\N\cup\{\infty\}$, let $u_n:D\to\C$ be the unique analytic function such that
$u_n(0)=1$ and $y_n(z):=z^{\nu+1}u_n(z)$ solves
\begin{equation}\label{oden}
 z^2y_n''=p_n(z) y_n.
 \end{equation}
Then $u_n(z)\to u_\infty(z)$ locally uniformly on $D$.
\end{lemma}

\begin{thm}\label{4:t1}
For every $n\in\N_0$ and $\nu\ge-\frac12$, we have
\begin{eqnarray}
 &&\hspace{-0.6cm}\tau^{-\nu-1}\F_\nu(K-\tau,\Lambda_\nu^n(k),k)
 \to \left(\frac{\sin \tau}{\tau}\right)^{\nu+1} \hyp21{-n,n+2\nu+2}{\nu+\tfrac32}{\sin^2 \tfrac{\tau}{2}}
 \label{4:limit}
 \end{eqnarray}
as $k\to 0$ locally uniformly for $|\Re\tau|<\pi$.
\end{thm}

\begin{proof}
The function $w(\tau):=\F_\nu(K-\tau,\Lambda_\nu^n(k),k)$ satisfies the differential equation
\begin{equation}\label{4:ode2}
  w''+\left(\Lambda_\nu^n(k)-\nu(\nu+1)\ns^2(\tau,k)\right) w=0 .
\end{equation}
It follows from Lemma \ref{2:l2} that $\Lambda_\nu^n(k)\to (n+\nu+1)^2$ as $k\to 0$.
Moreover,
\begin{equation}\label{4:k0}
\sn(z,k)\to\sin z,\, \cn(z,k)\to \cos z,\, \dn(z,k)\to 1\quad\text{as $k\to 0$}
\end{equation}
locally uniformly on $\C$ \cite[(22.5.3)]{NIST:DLMF}.
Now \eqref{4:k0}
and the maximum principle for analytic functions give
\[ \frac{\tau}{\sn(\tau,k)}\to \frac{\tau}{\sin \tau}\quad\text{as $k\to 0$ locally uniformly for $|\Re \tau|<\pi$.}\]
Therefore, we can apply Lemma \ref{4:l3} to the differential equations \eqref{4:ode2} with the limit differential equation
\begin{equation}\label{4:ode3}
 v''+\left((n+\nu+1)^2-\nu(\nu+1)\csc^2\tau\right) v=0 .
\end{equation}
The right-hand side of \eqref{4:limit} tends to $1$ as $\tau\to0$, and when multiplied by $\tau^{\nu+1}$ it is a solution of \eqref{4:ode3} belonging to the exponent $\nu+1$ at $\tau=0$. Therefore, \eqref{4:limit} follows from Lemma \ref{4:l3}.
\end{proof}

Gegenbauer polynomials $C_n^{\lambda}(x)$ are given by \cite[(15.9.2)]{NIST:DLMF}
\[ C_n^{\lambda}(x)=\frac{(2\lambda)_n}{n!} \hyp21{-n,2\lambda+n}{\lambda+\tfrac12}{\frac{1-x}{2}} .\]
Therefore, the limit in \eqref{4:limit} can be expressed in terms of Gegenbauer polynomials
\[ \hyp21{-n,n+2\nu+2}{\nu+\tfrac32}{\sin^2 \tfrac{\tau}{2}}= \frac{n!}{(2\nu+2)_n}C_n^{\nu+1}(\cos\tau).\]
Since the convergence is uniform for $\tau\in[0,\pi-\delta]$ for every $\delta>0$, we obtain
for $w(\tau):=\F_\nu(K-\tau,\Lambda_\nu^n(k),k)$
\begin{eqnarray*}
&& \hspace{-1.3cm}\int_{-K}^K w(\tau)^2\,{\mathrm d}\tau=2\int_0^K w(\tau)^2\,{\mathrm d}\tau \\
 &&\hspace{0.5cm} \to  \left(\frac{n!}{(2\nu+2)_n}\right)^2
\int_{-1}^1 (1-x^2)^{\nu+\frac12} \left(C_n^{\nu+1}(x)\right)^2{\mathrm d}x \quad\text{as $k\to0$} .
\end{eqnarray*}
Define $e_\nu^n$ such that
\[ \hspace{0.5cm}e_\nu^n:=\int_{-1}^1 (1-x^2)^{\nu+\frac12}\left(C_n^{\nu+1}(x)\right)^2\,{\mathrm d}x =\frac{\pi}{2^{2\nu+1}n!}
\frac{\Gamma(n+2\nu+2)}{(n+\nu+1)
\Gamma(\nu+1)^2},\]
whose value follows from \cite[(4.7.15)]{Szego}.
Therefore, Theorem \ref{4:t1} implies the following result.

\begin{thm}\label{4:t2}
For every $n\in\N_0$ and $\nu\ge-\frac12$, we have
\[ \tau^{-\nu-1}W_\nu^n(K-\tau,k)\to \left(e_\nu^n\right)^{-1/2} \left(\frac{\sin \tau}{\tau}\right)^{\nu+1} C_n^{\nu+1}(\cos\tau) \text{ as $k\to 0$}\]
locally uniformly for $|\Re\tau|<\pi$.
\end{thm}

In the application to peanut harmonics we are interested in the special case $\nu=m-\frac12$ with $m\in\N_0$.
The Ferrers function of the first kind $\PP_\nu^\mu$ satisfies the identity \cite[(18.11.1)]{NIST:DLMF}
\[ \PP_{m+n}^m(x)=\left(-\tfrac12\right)^m\frac{(2m)!}{m!} (1-x^2)^{m/2}C_n^{m+\frac12}(x),\]
where $m,n\in\N_0$.
Also using the duplication formula \cite[(5.5.5)]{NIST:DLMF}
\[ \Gamma(m+\tfrac12)\Gamma(m+1)=2^{-2m}\sqrt{\pi}\, \Gamma(2m+1) \]
we obtain the following corollary.

\begin{cor}\label{c1}
For $m,n\in\N_0$ we have
\[ \tau^{-m-\frac12}W_{m-\frac12}^n(K-\tau,k)\to \left(\frac{(m+n+\frac12)n!}{(2m+n)!}\right)^{1/2} (-\tau)^{-m} \left(\frac{\sin\tau}{\tau}
\right)^{1/2}\PP^m_{m+n}(\cos\tau) \]
as $k\to0$ locally uniformly for $|\Re\tau|<\pi$.
\end{cor}

We now determine the limit of Lam\'e-Wangerin functions $W_\nu^n$ on the imaginary axis.
Note that $W_\nu^n$ takes on real values on this line if $n$ is even and purely imaginary values if $n$ is odd.
We recall \cite[Lemma 6.1]{BiCohlVolkmerA}.

\begin{lemma}\label{4:l1}
(a)
Let $a\in\R$ and let $\{b_n\}$ be a sequence of real numbers such that $a<b_n\to \infty$ as $n\to\infty$.\\
(b)
For $n\in\N$, let $p_n,q_n:[a,b_n]\to\R$ be continuous functions such that $q_n(x)<0$ for all $x\in[a,b_n]$.
For every $n\in\N$, let $y_n:[a,b_n]\to\R$ be a nontrivial solution of the differential equation
\begin{equation}\label{node}
y_n''+p_n(x)y_n'+q_n(x)y_n =0
\end{equation}
such that $y_n(b_n)=0$. \\
(c)
Let $p_\infty, q_\infty:[a,\infty)\to\R$ be continuous functions such that
$p_n(x)\to p_\infty(x)$ and $q_n(x)\to q_\infty(x)$ as $n\to\infty$ uniformly on each compact interval $[a,b]$. Suppose that the differential equation
\begin{equation}\label{limitode}
 y_\infty''+p_\infty(x)y'_\infty+q_\infty(x)y_\infty=0
\end{equation}
admits a bounded nontrivial solution $y_\infty:[a,\infty)\to\R$, and that every solution of \eqref{limitode} which is linearly independent of $y_\infty$
is unbounded as $x\to\infty$.

Under assumptions (a), (b), (c), we have
\begin{equation}\label{conv1}
 \frac{y_n(x)}{y_n(a)}\to \frac{y_\infty(x)}{y_\infty(a)} \text{ and} \quad \frac{y_n'(x)}{y_n(a)}\to \frac{y'_\infty(x)}{y_\infty(a)}
\end{equation}
as $n\to\infty$ uniformly on every compact interval $[a,b]$.
The same result is true if the condition $y_n(b_n)=0$ is replaced by $y_n'(b_n)=0$.
\end{lemma}

We also use the following well-known lemma.

\begin{lemma}\label{4:l2}
Let $D$ be a simply-connected domain in $\C$, and $a\in D$. For $n\in\N$, let $p_n,q_n,p_\infty,q_\infty: D\to\C$ be analytic functions such that
$p_n(z)\to p_\infty(z)$ and $q_n(z)\to q_\infty(z)$ locally uniformly for $z\in D$. For each $n\in\N$ let $y_n: D\to\C$ be a solution
of the differential equation
\[ y_n''+p_n(z)y_n'+q_n(z)y_n=0 ,\]
and let $y_\infty:D\to\C$ be a solution of
\[ y_\infty''+p_\infty(z)y_\infty'+q_\infty(z)y_\infty=0. \]
If
\[ y_n(a)\to y_\infty(a),\, y_n'(a)\to y_\infty'(a)\quad \text{as $n\to\infty$},\]
then
\[ y_n(z)\to y_\infty(z) \quad\text{ as $n\to\infty$ locally uniformly for $z\in D$.}\]
\end{lemma}

\begin{thm}\label{4:t4}
For $n\in\N_0$ and $\nu\ge-\frac12$, we have
\begin{equation}\label{limit2}
 \frac{W_\nu^n(i(K'-\sigma),k)}{W_\nu^n(iK',k)}\to \expe^{-(n+\nu+1)\sigma}\quad \text{as $k\to 0$}
\end{equation}
locally uniformly for $|\Im \sigma|<\frac12\pi$.
\end{thm}
\begin{proof}
The function $w(\sigma)=W_\nu^n(i(K'-\sigma),k)$ satisfies the differential equation
\begin{equation}\label{ode4}
 \frac{{\mathrm d}^2w}{d\sigma^2}+q(\sigma,k)w=0,
\end{equation}
where
\[ q(\sigma,k):=-\Lambda_\nu^n(k)+\nu(\nu+1)k^2\nd^2(\sigma,k') .\]
It follows from Lemma \ref{2:l2} and \eqref{2:lowerbound} that $q(\sigma,k)<0$ for all $\sigma\in\R$.
Moreover, $\nd(\sigma,k')\to\cosh \sigma$ as $k\to 0$ locally uniformly for $|\Im \sigma|<\frac12\pi$, so
\begin{equation}\label{4:limit3}
 q(\sigma,k)\to -(n+\nu+1)^2\quad \text{as $k\to 0$}
\end{equation}
locally uniformly for $|\Im \sigma|<\frac12\pi$.
We have $w(K')=0$ if $n$ is even and $w'(K')=0$ if $n$ is odd.
Since $K'(k)\to\infty$ as $k\to0$,  we can apply Lemma~\ref{4:l1} (with $a=0$) and obtain
\eqref{limit2} and its differentiated form with uniform convergence for $\sigma\in[0,1]$. Local uniform convergence for
$|\Im \sigma|<\frac12\pi$ follows from Lemma \ref{4:l2}.
\end{proof}

\subsection{The peanut expansion of the $1/r$ potential in the limit $k\to 1$}

In spherical coordinates we have internal spherical harmonics
\[G_n^m(\r)=r^n \PP_n^m(\cos\theta) \expe^{im\phi}\]
and external spherical harmonics
\[H_n^m(\r)=r^{-n-1} \PP_n^m(\cos\theta) \expe^{im\phi},\]
where $\PP_n^m$ is the
Ferrers function of the first kind
(associated Legendre function of the first kind on-the-cut)
\cite[(14.3.1)]{NIST:DLMF} with
integer degree $n$ and integer order $m$.
Spherical harmonics are harmonic functions (solutions $u(\r)$ of Laplace's equation $-\Delta u=0$) of $\r=(x,y,z)$ expressed in spherical coordinates.
The functions $G_n^m$ are harmonic on $\R^3$ whereas the functions $H_m^n$ are harmonic on $\R^3\setminus\{\0\}$.
Let $\r,\r^\ast\in\R^3$ be two points with $r<r^\ast$. Then we have the well-known multipole expansion
of the fundamental solution of Laplace's
equation
which is crucial in a great many applications in mathematical physics
\cite[p.~1273-1274, (10.3.37)]{MorseFesh}
\begin{eqnarray}\label{1:expansion}
&&\hspace{-2.5cm}\frac{1}{\|\r-\r^\ast\|}= \sum_ {n=0}^{\infty} \sum_{m=-n}^{n} \frac{(n-m)!}{(n+m)!} G_n^m(\r)\overline{H_{n}^m(\r^\ast)}
\\
&&\hspace{-0.90cm}=\sum_{n=0}^\infty
\frac{r^n}{(r^\ast)^{n+1}}
\sum_{m=-n}^n
\frac{(n-m)!}{(n+m)!}
{\sf P}_n^m(\cos\theta)
{\sf P}_n^m(\cos\theta^\ast)\expe^{im(\phi-\phi^\ast)}
\label{git}\\
&&\hspace{-0.90cm}=\sum_{n=0}^\infty
\frac{r^n}{(r^\ast)^{n+1}}
P_n(\cos\theta\cos\theta^\ast+\sin\theta\sin\theta^\ast\cos(\phi-\phi^\ast)),
\label{Laplace}
\end{eqnarray}
where $P_n={\sf P}_n^0$ is the Legendre polynomial, and
\eqref{Laplace} is the famous Laplace expansion
of the $1/r:=1/\|\r-\r^\ast\|$ potential.
This expansion can be written in the form 
\begin{equation}\label{exp2}
\frac{1}{\|\r-\r^\ast\|}=\sum_{m\in\Z}\expe^{im(\phi-\phi^\ast)}\sum_{n=0}^\infty B_{m,n}(r,r^\ast,\theta,\theta^\ast),
\end{equation}
where
\begin{equation}
B_{m,n}=\frac{r^{|m|+n}}{(r^\ast)^{|m|+n+1}}\frac{n!}{(2|m|+n)!}
\PP^{|m|}_{|m|+n}(\cos \theta)\PP^{|m|}_{|m|+n}(\cos \theta^\ast),
\end{equation}
and $r,\theta,\phi$ are spherical coordinates of $\r$ and $r^\ast,\theta^\ast,\phi^\ast$ are spherical coordinates of $\r^\ast$ provided $r<r^\ast$.
 
Returning to flat-ring coordinates, if we substitute \eqref{3:intpeanut}, \eqref{3:extpeanut} and \eqref{4:R} in \eqref{3:expansion2}, we obtain
\[ \frac{1}{\|\r-\r^\ast\|} =\sum_{m\in\Z}\expe^{im(\phi-\phi^\ast)}\sum_{n=0}^\infty A_{m,n}(\sigma,\sigma^\ast,\tau,\tau^\ast,k), \]
where
\begin{eqnarray*}
&&\hspace{-2.1cm}A_{m,n}=2 \left(\frac{1-\sn(\sigma,k)\dn(\tau,k')}{\dn(\sigma,k)\sn(\tau,k')} \right)^{1/2}\left(\frac{1-\sn(\sigma^\ast,k)\dn(\tau^\ast,k')}{\dn(\sigma^\ast,k)\sn(\tau^\ast,k')} \right)^{1/2}\\[0.25cm]
&&\hspace{1cm}\times W_{|m|-\frac12}^n(K'-\tau,k')W_{|m|-\frac12}^n(K'-\tau^\ast,k')\\
 &&\hspace{1cm}\times\frac{1}{w_{m}^n(k)} W_{|m|-\frac12}^n(i(K+\sigma),k')W_{|m|-\frac12}^n(i(K-\sigma^\ast),k'),
\end{eqnarray*}
where $\sigma=s-K$, $\sigma^\ast=s^\ast-K$, $s<s^\ast$, $\tau=K'-t$, $\tau^\ast=K'-t^\ast$.

We now prove the main result of this section.
\begin{thm}\label{4:t5}
Let $m\in\Z$, $n\in\N_0$, $\tau,\tau^\ast\in(0,\pi)$, $\sigma, \sigma^\ast\in \R$.
Then
\[ A_{m,n}(\sigma,\sigma^\ast,\tau,\tau^\ast,k)\to B_{m,n}(\expe^\sigma,\expe^{\sigma^\ast},\tau,\tau^\ast)\quad\text{as $k\to1$}.
\]
\end{thm}
\begin{proof}
It is enough to consider $m\in\N_0$.
By \eqref{Rlimit} we have
as $k\to 1$,
\[  \left(\frac{1-\sn(\sigma,k)\dn(\tau,k')}{\dn(\sigma,k)\sn(\tau,k')} \right)^{1/2}\to  \expe^{-\frac12\sigma}(\sin\tau)^{-1/2}.\]
By Corollary \ref{c1}, we have
as $k\to 1$ that
\begin{eqnarray*}
&& \hspace{-1.5cm}W_{m-\frac12}^n(K'-\tau,k')W_{m-\frac12}^n(K'-\tau^\ast,k')\\
&&\hspace{0.7cm}\to \frac{(m+n+\frac12) n!}{(2m+n)!} (\sin\tau)^{1/2}\PP_{m+n}^m(\cos\tau) (\sin\tau^\ast)^{1/2}\PP_{m+n}^m(\cos\tau^\ast).
\end{eqnarray*}
By Theorem \ref{4:t4}, we have
as $k\to1$ that
\[ \frac{2}{w_{m}^n(k)} W_{m-\frac12}^n(i(K+\sigma),k')W_{m-\frac12}^n(i(K-\sigma^\ast),k')\to \frac{\expe^{(m+n+\frac12)\sigma} \expe^{-(m+n+\frac12)\sigma^\ast}}{m+n+\frac12}.\]
 Combining our results we obtain the statement of the theorem.
\end{proof}

This proves our assertion that the peanut expansion of the $1/r$ potential in the limit as $k\to 1$ becomes the famous multipole expansion of the $1/r$ potential in spherical coordinates.


\def\cprime{$'$} \def\dbar{\leavevmode\hbox to 0pt{\hskip.2ex \accent"16\hss}d}

\end{document}